\documentclass[10pt]{amsart}

\usepackage{graphicx}

\allowdisplaybreaks
\usepackage{enumerate,enumitem,amssymb} 
\usepackage{mathrsfs} 
\usepackage{tikz} 
\usepackage{esint} 

\newtheorem{theorem}{Theorem}[section]
\newtheorem{lemma}[theorem]{Lemma}

\theoremstyle{definition}
\newtheorem{definition}[theorem]{Definition}
\newtheorem{corollary}[theorem]{Corollary}

\newtheorem{proposition}[theorem]{Proposition}

\theoremstyle{remark}
\newtheorem{remark}[theorem]{Remark}

\numberwithin{equation}{section}

\newcounter{properties}

\DeclareMathOperator{\tr}{tr}

\DeclareMathOperator{\dv}{div}
\DeclareMathOperator{\grad}{grad}
\DeclareMathOperator{\Ric}{Ric}
\DeclareMathOperator{\inj}{Inj\,}
\DeclareMathOperator{\vol}{vol\,}
\DeclareMathOperator{\Sec}{Sec\,}

\begin{document}

\title[Convexity Properties of Harmonic Functions]{Convexity Properties of Harmonic Functions on Parameterized Families of Hypersurfaces}


\author[S.~M.~Berge]{Stine Marie Berge}
\address{Department of Mathematical Sciences, Norwegian University of Science and Technology, 7491 Trondheim, Norway}
\email{stine.m.berge@ntnu.no}



\date{}

\dedicatory{}

\commby{}

\begin{abstract}
   It is known that the $L^{2}$-norms of a harmonic function over spheres satisfies some convexity inequality strongly linked to the Almgren's frequency function. We examine the $L^{2}$-norms of harmonic functions over a wide class of evolving hypersurfaces. More precisely, we consider compact level sets of smooth regular functions and obtain a differential inequality for the $L^{2}$-norms of harmonic functions over these hypersurfaces. To illustrate our result, we consider ellipses with constant eccentricity and growing tori in $\mathbf{R}^3.$ Moreover, we give a new proof of the convexity result for harmonic functions on a Riemannian manifold when integrating over spheres. The inequality we obtain for the case of positively curved Riemannian manifolds with non-constant curvature is slightly better than the one previously known.\end{abstract}

\maketitle

\section{Introduction}
Since the paper by Almgren \cite{Al79}, the frequency function have been intensively used to study harmonic functions in $\mathbf{R}^n$ and, more generally, solutions to second order elliptic equations. For a harmonic function $h$ on $\mathbf{R}^n$ we let $H(t)$ denote the $L^2$-norm of $h$ over the sphere of radius $t$. In \cite{Ag66}, and later in  \cite{Al79}, it was shown that the function $H$ is \textit{geometrically convex}, i.e. 
\begin{equation}
   \label{eq:19}
H\left( r^{\alpha}s^{1-\alpha} \right)\le H\left( r\right)^{\alpha} H\left( s \right)^{1-\alpha}, \qquad 0 \leq \alpha \leq 1, \quad r,s > 0.
\end{equation}
Inequality \eqref{eq:19} is equivalent to the statement that the \textit{frequency function} $$N(t) = \frac{tH'(t)}{H(t)},\quad t>0$$ is increasing. The notion of frequency function was generalized to solutions of elliptic operators of divergence form by Garafalo and Lin in \cite{Ga86} and was shown to be almost increasing for $t<t_0$. They further used the result to show that the squares of solutions of the elliptic equations are Muckenhoupt weights on the ball $B_{R}$ with radius $R > 0$. \par 
In the paper of Mangoubi \cite{Ma13}, a more explicit convexity result on Riemannian manifolds was obtained by using comparison geometry. Using this result and extending eigenfunctions to harmonic functions, Mangoubi gave a new proof that a solution $u$ to  $\dv\left( \grad u \right)=-k^{2}u$ satisfies 
\begin{equation}
   \label{eq:17}
   \max_{B_{r}\left( p \right)} |u|\le C_1e^{C_2rk}\left(\max_{B_{2r}\left( p \right)} |u|\right)^{\alpha}\left(\max_{B_{r/2}\left( p \right)} |u|\right)^{1-\alpha}.
\end{equation}
In \eqref{eq:17} the positive constants $C_1$, $C_2$ and $0<\alpha<1$ only depend on the dimension and curvature of the Riemannian manifold. Inequality \eqref{eq:17} was first shown by Donnelly and Fefferman in \cite{DF88}. 

The main aim of this work is to study the $L^{2}$-norm of harmonic functions over families of surfaces, generalizing the geometric convexity inequality \eqref{eq:19}. Let $h$ be a harmonic function on a domain $\Omega$ in a Riemannian manifold $\left( \mathcal{M},\mathbf{g} \right)$ and fix a point $p \in \mathcal{M}$. Consider for $R > 0$ a smooth function $f:\Omega\to\left[ 0,R \right)$ that is regular and have compact level surfaces. Let $$H\left( t \right)=\int_{f^{-1}\left( t \right)}h^{2}\left|\grad f\right|\sigma_t$$ be the squared $L^{2}$-norm of $h$ over the level surface $f^{-1}\left( t \right)$ with the weight $\left|\grad f\right|$.\par
Our main theorem states that $H$ satisfies an inequality of the type 
\begin{equation}
\label{eq:27}
 \left( \log H\left( t \right) \right)''+\tau\left( t \right)\left( \log H\left( t \right) \right)'\ge \rho\left( t \right),  
\end{equation}
 where $\tau$ and $\rho$ are independent of $h$. In fact, the functions $\tau$ and $\rho$ only depend on explicit estimates on the derivatives of $f$ and are given in Theorem \ref{thm:1}. These kinds of inequalities when integrated imply that $H$ satisfies a variant of Inequality \eqref{eq:19}. In particular, when $f$ is the Riemannian distance function from a fixed point, we give a new proof of \cite[Theorem 2.2]{Ma13}. For the case when the curvature is positive we obtain a slight improvement of his inequality, see Section \ref{sec:sphere}.\par
Next we illustrate our result by considering \textit{$1$-homogeneous functions}, that is, functions that satisfy $f\left( t\mathbf{x} \right)=tf\left( \mathbf{x} \right)$ for $t>0$. A way to construct $1$-homogeneous functions is to choose a compact and star convex (with respect to the origin) set $R$ with smooth boundary $S \subset \mathbf{R}^{n} \setminus \{0\}$. Define a function $f$ by $f\left( \mathbf{x} \right)=1$ for all $\mathbf{x}\in S$ and extend this to a $1$-homogeneous function on the whole $\mathbf{R}^{n} \setminus \{0\}.$ In this case, we will show that there exist constants $A$ and $B$ such that the function $H$ satisfies 
\begin{equation}\label{eq:24}
   \left( \log H\left( t \right) \right)''+\frac{A}{t}\left( \log H\left( t \right) \right)'\ge -\frac{B}{t^{2}}.
\end{equation}
For the special case where $S$ is an ellipsoid in $\mathbf{R}^{n}$ we find in Section \ref{ell} the explicit values of $A$ and $B$. \par
To give an example of level surfaces not diffeomorphic to the sphere we take the distance function of the submanifold 
$$S^{k}\times\left\{ 0 \right\}\times\cdots\times \left\{ 0 \right\}\subset \mathbf{R}^{n},\textrm{  where  }S^{k}=\left\{ \mathbf{x} \in \mathbf{R}^{k+1}:\sqrt{x^2_1+\cdots+x^{2}_{k+1}} = 1\right\}.$$ In particular, whenever $k=1$ and $n=3$ the level surfaces form a family of tori. Let $f\left( \mathbf{x} \right)=\textrm{dist}\left( \mathbf{x},S^{k} \right)$, and $H$ be as above. Then for a fixed $0<\varepsilon<1$ we have that for all $t<1-\varepsilon$ the function $H$ satisfies \eqref{eq:27} for some $A$ and $B$.
Lastly, in Section \ref{sec:eig} we show that if $\dv \left( \grad u \right)=k^{2}u$ then the spherical $L^{2}$-norm of $u$ satisfies \eqref{eq:19}.

\subsection*{Acknowledgment}
The author would like to thank Eugenia Malinnikova for her guidance and Dan Mangoubi for his insightful suggestions. The author was partially supported by the BFS/TFS project Pure Mathematics in Norway.
\section{The Convexity Result}
\subsection{Prerequisites}\label{sec:2}
In this article $\left( \mathcal{M},\mathbf{g} \right)$ will always denote a smooth Riemannian manifold. The volume density and its respective divergence will be denoted by $\vol$ and $\dv$. We will use the notation $\nabla$ to denote the Levi-Civita connection, and define the Hessian of a function $h\in C^{\infty}\left( \mathcal{M} \right)$ by $$\nabla^{2}h\left( X,Y \right)= \nabla_{X}\nabla_{Y}h-\nabla_{\nabla_XY}h=\langle \nabla_{X}\grad h,Y\rangle,$$ where $X$ and $Y$ are vector fields and $\grad h$ denotes the gradient of the function $h$. The Laplace operator $\Delta$ is given by 
$$\Delta h=\dv\left( \grad h \right)=\tr_{\mathbf{g}}\nabla^{2}h\left( \times,\times \right),$$
where $\tr_{\mathbf{g}}$ denotes the trace with respect to the metric $\mathbf{g}.$

The idea of the proof of Theorem \ref{thm:1} is to emulate the proof of the well known special case (which is presented in details in Section \ref{sec:sphere}): Let $h:B_R\left( p \right)\subset \mathcal{M}\to \mathbf{R}$ be a harmonic function on the ball with radius $R$ centered at $p$ and define $$H(t)=\int_{S_t}h^{2}\sigma_t,$$ where $S_t$ is the geodesic sphere centered at $p$ with radius $t$ and $\sigma_t$ is its surface measure. In \cite{Ma13} it was shown that $H$ satisfies a convexity property, which in the case of constant curvature spaces is on the form 
$$\left( \log H\left( t \right) \right)''+\log\left( \sin_\mathcal{K}\left( t \right) \right)'\left( \log H\left( t \right) \right)'\ge -\left( n-1 \right)\mathcal{K}+\left( n-2 \right)\min\left( 0,\mathcal{K} \right),$$ where $\mathcal{K}$ is the sectional curvature and $\sin_{\mathcal{K}}\left( r \right)$ is the function defined by Equation \eqref{sin} in Section \ref{sec:sphere}.

Our goal is to obtain a similar inequality for other families of parameterized surfaces than geodesic spheres. Since an important step in \cite{Ma13} depends indirectly on the coarea formula, we will assume that this family of surfaces is given as the level surfaces of a Lipschitz function $f:\Omega\to \left[ 0,R \right),$ where $\Omega\subset \mathcal{M}$ is an open set. To ensure that the preimages $S_t=f^{-1}\left( t \right)$ are hypersurfaces for $t\in \left( 0,R \right)$, we will assume that every value in $\left( 0,R \right)$ is regular (see \cite[Theorem 5.12]{Le13}), meaning that $\left|\grad f\right|>0$ for all $x\in f^{-1}\left( 0,R \right).$ 
We will also need that the integral over the hypersurfaces are finite. Thus we assume that that the surfaces $S_{t}$ are closed manifolds, that is, compact manifolds without boundary. Finally, to be able to use the divergence theorem on the surfaces $S_t$ we will assume that $R_t=f^{-1}\left( \left[ 0,t \right) \right)$ is open and compactly embedded in $\Omega$ for all $t\in \left( 0,R \right)$ and $S_t$ is given as the boundary of $R_{t}.$
\begin{definition}
We say that the function $f:\Omega \to [0,R)$ is a \textit{parameterizing function} if it satisfies the following properties;
\begin{enumerate}
   \item $f$ is Lipschitz continuous in $\Omega$ and smooth on $f^{-1}\left( \left( 0,R \right) \right)$,\label{pro:1}
   \item all values in $\left( 0,R \right)$ are regular values of $f$, and $S_t=f^{-1}\left( t \right)$ are closed hypersurfaces in $\mathcal{M},$ \label{pro:2}
   \item $R_t=f^{-1}\left( [0,t) \right)$ are compactly embedded submanifolds of $\Omega$ with boundary $S_{t}$. Furthermore, we need that $\frac{1}{\left|\grad f\right|}$ is an integrable function on $R_{t}$ for all $t$.\label{pro:3}
\setcounter{properties}{\value{enumi}}
\end{enumerate}
\end{definition}
Under the assumptions on the function $f$ we can formulate the coarea formula on manifolds.
\begin{lemma}[{\cite[Theorem 3.1]{Fe59}}]
   \label{thm:2}
   Let $f:\Omega\subset \mathcal{M}\to \mathbf{R}$ be a parameterizing function. Define $S_{t}=f^{-1}\left( t \right)$ and let $\sigma_t$ be the area measure on $S_{t}$. Then for all integrable functions $\varphi:\Omega\to \mathbf{R}^n$ we have that 
   $$\int_{R_t}\varphi\vol=\int_{0}^{t}\int_{S_{s}}\frac{\varphi}{|\grad f|}\sigma_s ds.$$
\end{lemma}

It will be beneficial for us to view $S_{t}$ as variations of hypersurfaces following the flow of the vector field $\frac{\grad f}{\left|\grad f\right|^{2}}.$ To make this precise, we formulate the following lemma.
\begin{lemma}
   \label{lem:8}
   Let $f:\Omega\to \left[0,R \right)$ be a parameterizing function. Let $\varphi_{t}$ denote the flow of the vector field $\frac{\grad f}{\left|\grad f\right|^{2}}$ and fix a value $t_0\in \left( 0,R \right)$. Then for all $t_0+t<R$ we have that $\varphi_{t}\left( S_{t_0} \right)=S_{t_0+t}$.
\end{lemma}
\begin{proof}
Let $\gamma\left( t \right)$ be an integral curve of $\frac{\grad f }{\left|\grad f\right|^{2}}$ such that $\gamma\left( 0 \right)=x\in S_{t_{0}}.$ We need to show that $f\left( \gamma\left( t \right) \right)=t_0+t.$ Taking the derivative we obtain 
\begin{equation}
   \label{eq:9}
   \frac{d}{dt}f\left( \gamma\left( t \right) \right)=d f\left( \dot\gamma\left( t \right) \right)=\langle \grad f,\dot \gamma \left( t \right)\rangle=\left\langle \grad f,\frac{\grad f}{\left|\grad f\right|^{2}}\right\rangle=1.
\end{equation}
Integrating \eqref{eq:9} shows that $\varphi_t\left( S_{t_0} \right)\subset S_{t+t_0}$. To see that $\varphi_t\left( S_{t_0} \right) = S_{t+t_0}$ we pick $p \in S_{t+t_0}$. Since $\varphi_{t}$ is a diffeomorphism with $\varphi_{t}^{-1}=\varphi_{-t}$ the element $\varphi_{-t}(p)$ is in $S_{t_{0}}$ by repeating the argument above. Hence $\varphi_{t}(\varphi_{-t}(p)) = p$ and the result follows. 
\end{proof}

We remind the reader that the \textit{Lie derivative} of a $k$-form $\omega$ in the direction of the vector field $X$ evaluated at the point $p\in M$ is given by $$\left( \mathcal{L}_{X}\omega \right)_p=\lim_{t\to 0}\frac{\varphi_{t}^*\left( \omega_{\varphi_t\left( p \right)} \right)-\omega_p}{t},$$ where $\varphi_t^{*}$ denotes the pull back with respect to the flow $\varphi_t$ of $X$. We will use some standard properties of the Lie derivative acting on forms which can be found in \cite[p.~372]{Le13}.
   Let $X$ be a $C^1$ vector field and $\omega$ and $\nu$ be differentiable $k$- and $l$-forms, respectively. Then 
   \begin{equation}
      \label{lem:9a} \mathcal{L}_X\left( \omega \wedge \nu\right)=\left( \mathcal{L}_X \omega \right)\wedge \nu+\omega\wedge\left( \mathcal{L}_{X}\nu \right),
   \end{equation}
   \begin{equation}
      \label{lem:9b}\text{\textit{Cartan's Magic Formula:}}\quad \mathcal{L}_{X}\omega=\iota_{X}d\omega+d\iota_{X}\omega,
   \end{equation}
   where $\iota$ denotes the interior product and $d$ is the exterior differential.
The Cartan's magic formula implies that for any function $f$ we have the formula
\begin{equation}
   \label{eq:12}
   \mathcal{L}_{fX}\omega=f\mathcal{L}_{X}\omega+df\wedge \left( \iota_{X}\omega \right),
\end{equation}
where we have used that $\iota_{fX}=f\iota_{X}$ and $d\left( f\omega \right)=df \wedge \omega+fd\omega.$
The reason for going from the level surfaces of a function to a variation of surfaces by using the flow point of view is to utilize the following differentiation theorem.
\begin{lemma}
   \label{lem:2}
   Let $\alpha$ be an $\left( n-1 \right)$-form and let $S$ be an oriented closed smooth hypersurface in $\mathcal{M}$. Denote by $X$ a vector field and denote by $\varphi_t$ the flow generated by $X.$ Then $$\frac{d}{dt}\int_{\varphi_t\left( S \right)}\alpha=\int_{\varphi_{t}\left( S \right)}\mathcal{L}_{X}\alpha,$$ where $\mathcal{L}_{X}$ denotes the Lie derivative with respect to $X.$
\end{lemma}
\begin{proof}
   By using the definition of $\frac{d}{dt}$ we get 
   \begin{align*}
      \frac{d}{dt}\int_{\varphi_{t}\left( S \right)}\alpha&=\lim_{h\to 0}\frac{\int_{\varphi_{t+h}\left( S \right)}\alpha-\int_{\varphi_{t}\left( S \right)}\alpha}{h}\\
      &=\lim_{h\to 0}\frac{\int_{\varphi_{t}\left( S \right)}\varphi^*_{h}\alpha-\int_{\varphi_{t}\left( S \right)}\alpha}{h}\\
      &=\int_{\varphi_{t}\left( S \right)}\lim_{h\to 0}\frac{\varphi^*_{h}\alpha-\alpha}{h}\\
      &=\int_{\varphi_{t}\left( S \right)}\mathcal{L}_{X}\alpha.
   \end{align*}
We refer the reader to \cite[p.139]{Fr12} where the result is proved for more general variations of submanifolds.
\end{proof}

\subsection{The Main Theorem}\label{sec:main}
Let $h:\Omega\to \mathbf{R}$ be a harmonic function where $\Omega \subset \mathcal{M}$ is an open set and let $f:\Omega\to \left[ 0,R \right)$ be a parametrization function from Section \ref{sec:2}. Define the function $$H\left( t \right)=\int_{S_{t}}h^{2}\left|\grad f\right|\sigma_t,$$ where $S_{t}=f^{-1}\left( t \right)$ and $\sigma_t$ is its surface measure. The goal of this section is to show that $H$ satisfies a convexity property.\par
We will need the following version of a result of H\"ormander, \cite[Theorem 1]{Ho18}: 
   Let $f$ be a parameterizing function and $S_{t}$ and $\sigma_t$ be as above. 
   Then there exists a function $K$ only depending on $f$ such that for any harmonic function $h,$
   \begin{equation}\label{eq:7}
   \int_{S_{t}}\frac{\left|\grad_{S_t}h\right|^{2}-h_{n}^{2}}{\left|\grad f\right|}\sigma_t\ge -K\left( t \right)\int_{R_{t}}\left|\grad h\right|^{2}\vol,
\end{equation}
   where $\grad_{S_{t}}h$ and $h_n$ denote the gradient with respect to $S_t$ and the unit normal derivative, respectively.
   Inequality \eqref{eq:7} is proved in the end of this section, see Lemma \ref{lem:4:proof}.
   The following theorem is the main result of the paper; it shows that for any harmonic function $h$ the $L^{2}$-norms $H$ satisfy some convexity inequality only depending on the function $f.$
\begin{theorem}
   \label{thm:1}
   Let $\left( \mathcal{M},\mathbf{g} \right)$ be a Riemannian manifold, and let the functions $h$, $f$ and $H$ be as described earlier in this section. Define the functions $m,\,M$ and $g$ such that:
   \begin{enumerate}
\setcounter{enumi}{\value{properties}}
      \item\label{eq:14} $m\left( t \right)\le \frac{\Delta f}{\left|\grad f\right|^{2}}\le M\left( t \right)$ on $S_{t}$, and
      \item\label{eq:13} $\left\langle \grad \left(\frac{\Delta f}{\left|\grad f\right|^{2}}\right),\frac{\grad f}{\left|\grad f\right|^{2}}\right\rangle\ge g\left( t \right)$ on $S_{t}.$ 
   \end{enumerate}
Then $H$ satisfies the growth estimate
\begin{align}\label{eq:5}
   H'\left( t \right)&=2\int_{S_{t}} hh_{n}\sigma_{t}+ \int_{S_{t}} h^{2}\frac{\Delta f }{\left|\grad f\right|^{2}} \left|\grad f\right|\sigma_{t}\\
   &\ge 2\int_{R_{t}}\left|\grad h\right|^{2}\vol+m\left( t \right)H\left( t \right).\nonumber
\end{align} 
Moreover, if $K$ is the function given in Inequality \eqref{eq:7} and $K\left( t \right)+M\left( t \right)\ge 0$ then 
\begin{equation}\label{eq:6}
   \left( \log H\left( t \right) \right)''+\left( K\left( t \right)+M\left( t \right) \right) \left( \log H\left( t \right) \right)' \ge g\left( t \right)+m\left( t \right) M\left( t \right)+m\left( t \right)K\left( t \right) .
\end{equation}
\end{theorem}
\begin{proof}
Using Lemma \ref{lem:2} to take the derivative of $H$ we obtain
\begin{align}\label{eq:8}
H'\left( t \right)&=\int_{S_{t}}\mathcal{L}_{\left( \grad f \right)/\left|\grad f\right|^{2}}\left( h^{2}\left|\grad f\right|\sigma_{t} \right)\\
&=\int_{S_{t}}\left\langle\frac{\grad f}{\left|\grad f\right|^{2}},\grad h^{2}\right\rangle\left|\grad f\right|\sigma_t\nonumber\\
&\quad + \int_{S_{t}}h^{2}\mathcal{L}_{\left( \grad f \right)/\left|\grad f\right|^{2}}\left( \left|\grad f\right|\sigma_{t} \right)\nonumber\\
&=2\int_{S_{t}}hh_n\sigma_t+ \int_{S_{t}}h^{2}\mathcal{L}_{\left( \grad f \right)/\left|\grad f\right|^{2}}\left( \left|\grad f\right|\sigma_{t} \right)\nonumber
\end{align}

The following lemma takes care of the last term in the above computation and finishes the proof of Equation \eqref{eq:5}. In the literature a version of the next lemma is known as the first variation of area for hypersurfaces (see \cite[p.~51]{CL06}).
   \begin{lemma}\label{lem:3}
   Using the notation above, we have that
\begin{align*}
   \mathcal{L}_{\grad f/|\grad f|^{2}}\left( \left|\grad f\right|\sigma_{t} \right)&=\frac{ \Delta f }{\left|\grad f\right|}\sigma_{t}\\
   &=\left( \left\langle\grad f,\grad \left|\grad f\right|\right\rangle/\left|\grad f\right|^{2}-nH_t \right)\sigma_t,
\end{align*}
where $H_t$ is the mean curvature of $S_{t}$ and $n$ is the dimension of $\mathcal{M}$.
\end{lemma}
\begin{proof}
   Using the properties of the Lie derivative given by Equation \eqref{lem:9a} and \eqref{lem:9b} together with the definition of the divergence we obtain 
   \begin{align*}
      \mathcal{L}_{\frac{\grad f}{|\grad f|^{2}}}\left(\left|\grad f\right| \sigma_{t} \right)&=\left\langle \frac{\grad f}{\left|\grad f\right|^{2}},\grad \left|\grad f\right|\right\rangle\sigma_{t}+\mathcal{L}_{\frac{\grad f}{|\grad f|}}\left( \sigma_{t} \right)\\ 
      &\qquad+\left|\grad f\right|d\left( 1/\left|\grad f\right| \right)\wedge \iota_{\frac{\grad f}{\left|\grad f\right|}}\left( \sigma_{t} \right)\\
      &=\left\langle \frac{\grad f}{\left|\grad f\right|^{2}},\grad \left|\grad f\right|\right\rangle\sigma_{t}+\mathcal{L}_{\frac{\grad f}{|\grad f|}}\left( \iota_{\frac{\grad f}{\left|\grad f\right|}} \vol \right)\\
      &=\left\langle \frac{\grad f}{\left|\grad f\right|^{2}},\grad \left|\grad f\right|\right\rangle\sigma_{t}+\iota_{\frac{\grad f}{\left|\grad f\right|}}\left( \mathcal{L}_{\frac{\grad f}{\left|\grad f\right|}}\vol \right)\\
   &=\left\langle \frac{\grad f}{\left|\grad f\right|^{2}},\grad \left|\grad f\right|\right\rangle\sigma_{t}+\dv\left(\frac{\grad f}{\left|\grad f\right|} \right)\sigma_t\\
   &=\frac{\Delta f }{\left|\grad f\right|}\sigma_{t}.
\end{align*}
\end{proof}
This concludes the proof of the Identity \eqref{eq:5}, note that the expression for $H'$ holds for an arbitrary function $h$ not necessarily harmonic. 

To prove the differential inequality \eqref{eq:6} we differentiate \eqref{eq:5}. We rewrite the first term by using the divergence formula and applying the coarea formula given in Lemma \ref{thm:2} when $\varphi\left( x \right)=h^{2}\left( x \right)$ and obtain $$D\left( t \right):=\int_{S_{t}} hh_{n}\sigma_{t}=\int_{R_{t}}\left|\grad h\right|^{2}\vol=\int_{0}^{t}\int_{S_{s}}\left|\grad h\right|^{2}\frac{1}{\left|\grad f\right|}\sigma_{s}ds.$$
Computing the second derivative of $H$ by applying Lemma \ref{lem:2} and \ref{lem:3} once more gives
\begin{align*}
   H''\left( t \right)&=2\int_{S_{t}} \frac{\left|\grad h\right|^{2}}{\left|\grad f\right|}\sigma_{t}+2 \int_{S_{t}} hh_{n}\frac{\Delta f }{\left|\grad f\right|^{3}} \left|\grad f\right|\sigma_{t} \\
   &\quad+\int_{S_{t}}h^{2}\left( \frac{\Delta f}{\left|\grad f\right|^{2}} \right)^{2}\left|\grad f \right|\sigma_{t}\\
   &\quad+\int_{S_{t}}h^{2}\left\langle \grad \left( \frac{\Delta f}{|\grad f|^{2}} \right),\frac{\grad f}{\left|\grad f\right|^{2}}\right\rangle\left|\grad f\right|\sigma_{t}.
\end{align*}
Using that $\left|\grad h\right|^{2}=2h_{n}^{2}+\left|\grad_{S_{t}}h\right|^2-h_{n}^{2}$ and denoting $$G\left( x \right)=\left\langle \grad \left( \frac{\Delta f}{|\grad f|^{2}} \right),\frac{\grad f}{\left|\grad f\right|^{2}}\right\rangle,$$ we have
\begin{align*}
   H''\left( t \right)&=2\int_{S_{t}} h_{n}^{2}/\left|\grad f\right|\sigma_{t}+2\int_{S_{t}}\left( \left|\grad_{S_{t}} h\right|^{2}-h_{n}^{2} \right)\frac{1}{\left|\grad f\right|}\sigma_{t}\\
   &\quad+\frac{1}{2}\int_{S_{t}}h^{2}\left( \frac{\Delta f}{\left|\grad f\right|^{2}} \right)^{2}\left|\grad f \right|\sigma_{t}+\int_{S_{t}}h^{2}G\left|\grad f\right|\sigma_{t}\\
   &\quad+2 \int_{S_{t}} \left( \frac{h_{n}}{\left|\grad f\right|}+\frac{h\Delta f }{2\left|\grad f\right|^{2}} \right)^{2} \left|\grad f\right|\sigma_{t}.
\end{align*}
Applying Inequality \eqref{eq:7} we get 
\begin{align}\label{eq:23}
   \int_{S_t}\frac{\left( \left|\grad_{S_t}h\right|^{2}-h_{n}^{2} \right)}{ \left|\grad f\right|}\sigma_{t}&\ge -K\left( t \right)\int_{S_t}hh_n\sigma_{t}=-K\left( t \right)D\left( t \right),
\end{align} 
and by Cauchy-Schwarz we have the inequalities
\begin{align}
   \label{eq:22}
   2\int_{S_{t}}\left( \frac{h_n}{\left|\grad f\right|}+\frac{h\Delta f}{2\left|\grad f\right|^{2}} \right)^{2}\left|\grad f\right|\sigma_t H\left( t \right)\ge \frac{1}{2}H'\left( t \right)^{2}
\end{align}
and 
\begin{align}
   \label{eq:21}
   \nonumber\frac{1}{2}H'\left( t \right)^{2}&-2D\left( t \right)\int_{S_{t}}\frac{h^{2}\Delta f}{\left|\grad f\right|^{2}}\left|\grad f\right|\sigma_{t}-\frac{1}{2}\left( \int_{S_{t}}h^{2}\frac{\Delta f}{\left|\grad f\right|^{2}}\left|\grad f\right|\sigma_t \right)^{2}\\
   & \le 2\int_{S_{t}} \frac{h_n^{2}}{|\grad f|}\sigma_tH\left( t \right).
\end{align}
A straightforward computation combining \eqref{eq:23}, \eqref{eq:22} and \eqref{eq:21} shows that 
\begin{align}
\label{eq:28}
   H''\left( t \right)H\left( t \right)-H'\left( t \right)^{2}&\ge -2K\left( t \right)D\left( t \right)H\left( t \right)-2D\left( t \right)\int_{S_t}h^{2}\frac{\Delta f}{\left|\grad f\right|^{2}}\left|\grad f\right|\sigma_{t}\\
   &+H\left( t \right)\int_{S_{t}}h^{2}\left\langle\grad\left(\frac{\Delta f}{ \left|\grad f\right|^{2}}\right) ,\frac{\grad f}{\left|\grad f\right|^{2}}\right\rangle\left|\grad f\right|\sigma_{t}.  \nonumber
\end{align}
Applying the estimates \eqref{eq:14} and \eqref{eq:13} and noting the fact that $K+M$ is non-negative, implies 
\begin{align*}
   H''\left( t \right)H\left( t \right)&-H'\left( t \right)^{2}+\left( K\left( t \right)+M\left( t \right) \right)H'\left( t \right)H\left( t \right)\\
   &\ge H\left( t \right)\int_{S_{t}}h^{2}\left\langle\grad\left(\frac{\Delta f}{ \left|\grad f\right|^{2}}\right) ,\frac{\grad f}{\left|\grad f\right|^{2}}\right\rangle\left|\grad f\right|\sigma_{t}\\
   &\quad+\left( M\left( t \right)+K\left( t \right) \right)H\left( t \right)\int_{S_t}h^{2}\frac{\Delta f}{\left|\grad f\right|^{2}}\left|\grad f\right|\sigma_{t}\\
   &\ge g\left( t \right)H\left( t \right)^{2}+m\left( t \right)\left( M\left( t \right)+K\left( t \right) \right)H\left( t \right)^{2}\\
   &\ge \left( g\left( t \right)+m\left( t \right) M\left( t \right)+m\left( t \right)K\left( t \right)  \right)H\left( t \right)^{2}.
\end{align*}
Dividing both sides of the equation by $H\left( t \right)^{2}$, we obtain \eqref{eq:6} and thus finish the proof of Theorem \ref{thm:1}.
\end{proof}
\begin{remark}
   \label{re:3}
   Sometimes it is beneficial to replace Inequality \eqref{eq:7} by 
   \begin{equation}
      \label{eq:7:n}
      \int_{S_{t}}\frac{ \left|\grad_{S_{t}}h\right|^{2}-h_n^{2} }{\left|\grad f\right|}\sigma_t\ge -K\left( t \right)D\left( t \right)+k\left( t \right)H\left( t \right),
   \end{equation}
   to obtain a better result.
   Using this inequality in the proof above replaces Inequality \eqref{eq:28} with
\begin{align*}
   H''\left( t \right)H\left( t \right)-H'\left( t \right)^{2}&\ge -2K\left( t \right)D\left( t \right)H\left( t \right)+2k\left( t \right)H\left( t \right)^{2}\\
   &\quad+H\left( t \right)\int_{S_{t}}h^{2}\left\langle\grad\left(\frac{\Delta f}{ \left|\grad f\right|^{2}}\right) ,\frac{\grad f}{\left|\grad f\right|^{2}}\right\rangle\left|\grad f\right|\sigma_{t}\\
   &\quad-2D\left( t \right)\int_{S_t}h^{2}\frac{\Delta f}{\left|\grad f\right|^{2}}\left|\grad f\right|\sigma_{t}.
\end{align*}
Completing the proof in the same manner as before gives 
\begin{equation*}
   \left( \log H\left( t \right) \right)''+\left( K\left( t \right)+M\left( t \right) \right) \left( \log H\left( t \right) \right)' \ge g\left( t \right)+m\left( t \right)\left( M\left( t \right)+K\left( t \right)\right)+2k\left( t \right)
\end{equation*} as a generalization of Inequality \eqref{eq:6} in Theorem \ref{thm:1}. 
We will use this modified version of Theorem \ref{thm:1} in Section \ref{sec:sphere} when the upper bound of the sectional curvature $\mathcal{K}$ is negative.
\end{remark}
\begin{remark}
   \label{re:1}
   If $\mathcal{M}$ is an oriented manifold, then $S_{t}$ is always orientable. In general, when $\mathcal{M}$ is orientable any hypersurface that can be described as the level set of a regular value of a smooth function is orientable (see \cite[Proposition 15.23]{Le13}).
\end{remark}

\subsection{Corollaries}
Before proving Inequality \eqref{eq:7}, we provide some corollaries and remarks.
\begin{corollary}
   \label{cor:1}
   Let $f:\Omega\to \left[ 0,\infty \right)$ be a convex and parameterizing function. Then $m$ is non-negative, and hence $H$ is increasing. In this case, the sets $R_{t}=f^{-1}\left( \left[ 0,t \right) \right)$ are (totally) convex.
\end{corollary}
   \begin{proof}
      That $f$ is \textit{convex} means that the Hessian of $f$ satisfies $\left( \nabla^{2}f \right)\left( v,v \right)\ge 0$ for all $v\in T_p\mathcal{M}$ and $p \in \mathcal{M}$. Taking the trace of the Hessian of $f$ shows that $\Delta f\ge 0$, and hence $m\left( t \right)\ge 0.$ Thus Inequality \eqref{eq:5} implies that $H$ is increasing. We say that $R_{t}$ is \textit{(totally) convex} if any geodesic in $\Omega$ starting and ending in $R_t$ is contained in $R_t$. For a geodesic $\gamma$ a straightforward computation gives that $$\frac{d^{2}}{ds^{2}}f\left( \gamma\left( s \right) \right)=\nabla^{2}f \left( \dot\gamma\left( s \right),\dot\gamma \left( s \right) \right)\ge 0.$$ Hence if $\gamma$ is geodesic such that $\gamma\left( 0 \right)=x\in R_{t}$ and $\gamma\left( 1 \right)=y\in R_{t}$, then $$f\left( \gamma\left(\lambda \right) \right)\le \lambda f\left( \gamma\left( 1 \right) \right)+\left( 1-\lambda \right)f\left( \gamma\left( 0 \right) \right)\le t.$$ In conclusion, we have that $\gamma \subset R_{t}$ and hence $R_{t}$ is (totally) convex.
   \end{proof}
   In the case when $\left|\grad f\right|$ is constant, the term $\frac{\Delta f}{\left|\grad f\right|}$ coincides with the mean curvature, giving a geometric interpretation to the functions $m,\,M$ and $g.$ When $f$ is given as the distance function from a compact submanifold (e.g. radial distance function) we have that $\left|\grad f\right|=1$ (see \cite[Theorem 6.38]{Le18}). Letting $f$ be the distance function from a point, then $m\left( t \right)=M\left( t \right)$ is equivalent with the Riemannian manifold $\left( \mathcal{M},\mathbf{g} \right)$ being \textit{locally harmonic at $p$}, meaning that $\fint_{S_{t}}h\sigma_t$ is constant for all $h$ and $t$ less than some fixed $\varepsilon.$ When $\left|\grad f\right|$ is not constant the geometric interpretation of $m,\, M$ and $g$ becomes somewhat more diffuse. However, the following proposition tells us that the difference $M\left( t \right)-m\left( t \right)$ measures how far the level sets of $f$ are from satisfying the mean value theorem.   \begin{proposition}
      \label{prop:1}
      Assume that $f:\Omega\to \mathbf{R}$ is a parameterizing function such that $\frac{\Delta f}{\left|\grad f\right|^{2}}=M\left( t \right)$ on $S_{t}=f^{-1}\left( t \right).$
      Then
      $$F\left( t \right)=\frac{\int_{S_{t}}h\left|\grad f\right|\sigma_{t}}{\int_{S_{t}}\left|\grad f\right|\sigma_t}$$
      satisfies the mean value property, i.e. $F'\left( t \right)=0.$

      When $f\left( x \right)=r\left( x \right)$ is the radial distance function centered at the point $p$, then $\Delta r=M\left( t \right)$ for all $t$ less that some fixed $\varepsilon$ is equivalent with the Riemannian manifold being locally harmonic at $p$.
 \end{proposition}
   \begin{proof}
 The derivative of $F$ is equal to $$F'\left( t \right)=\frac{\int_{S_{t}}h\frac{\Delta f}{\left|\grad f\right|^{2}}\left|\grad f\right|\sigma_{t}}{\int_{S_{t}}\left|\grad f\right|\sigma_t}-\frac{\int_{S_{t}}\frac{\Delta f}{\left|\grad f\right|^{2}}\left|\grad f\right|\sigma_{t}}{\int_{S_{t}}\left|\grad f\right|\sigma_t }F\left( t \right),$$
         by using Lemma \ref{lem:3}.
         Hence if $\frac{\Delta f}{\left|\grad f\right|^{2}}=M\left( t \right)$ we get that $F'\left( t \right)=0$, and $F\left( t \right)$ is constant.

         For the last claim we utilize that the manifold is locally harmonic if and only if the geodesic spheres centered at $p$ have constant \textit{mean curvature} (see \cite[Proposition 3.1.2]{Kr10}). The mean curvature $H_t$ of a hypersurface given as a level surface of a function $f$ at the value $t$ satisfies
         \begin{align*}
            H_t&=- \frac{\Delta f-\nabla^{2}f\left( \frac{\grad f}{|\grad f|}, \frac{\grad f}{|\grad f|}\right) }{n\left|\grad f\right|}\\
            &=-\frac{1}{n}\dv \left(\frac{\grad f}{\left|\grad f\right|}  \right),
         \end{align*} (see \cite[Exercise 8-2 b)]{Le18}). 
         Since the gradient of $r$ has norm one we get $H_{t}=-\frac{1}{n}\Delta f$, which proves the claim.
   \end{proof}
\subsection{An Inequality of H\"ormander}
The only thing left is to prove Inequality \eqref{eq:7}. The statement and its proof can be found in H\"ormander's works, see \cite[p.~38 Theorem 1]{Ho18}. We need a weighted version of the inequality and provide a proof for the convenience of the reader. 
\begin{lemma}\label{lem:4:proof}
   Let $\Delta h=0$, and assume that $\mathcal{R}$ is an open compactly embedded manifold. Denote by $S:=\partial \mathcal{R}$ and by $\sigma$ its area measure. Let $\partial_n$ denote any smooth extension of the outward unit normal vector of $S$ to $\mathcal{R}$.
   Then 
\begin{align}\label{eq:3}
   \int_{S}&\left( \left|\grad_{S}h\right|^{2}-h_{n}^{2} \right) w\sigma\\
   &=\int_{\mathcal{R}}|\grad h|^{2}\left( \dv\left( w\partial_n \right)-2\left\langle \nabla_{\frac{\grad h}{\left|\grad h\right|}}\left( w\partial_n \right),\frac{\grad h}{\left|\grad h\right|}\right\rangle  \right)\vol,\nonumber
\end{align} 
where $w\left( x \right)\partial_n$ is a smooth vector field defined on $\mathcal{R}$. Since $\mathcal{R}$ is compact there exists a minimum (and maximum) of $$\dv\left( w\left( x \right)\partial_n \right)-2\langle \nabla_{v}\left( w\left( x \right)\partial_n \right),v\rangle$$
where $v\in T_{x}\mathcal{M}$ and $\left|v\right|=1.$
Hence there exists a constant $K$ such that
\begin{align*}
   \int_{S}\left( \left|\grad_{S}h\right|^{2}-h_{n}^{2} \right) w\sigma&\ge -K\int_{S}hh_n\sigma.
\end{align*} 
\end{lemma}
\begin{proof}
   Denote by $X=w\left( x \right)\partial_n$ and $V=2X\left( h \right)\grad h-\left|\grad h\right|^{2}X$. Then computing the divergence of $V$ we get 
   \begin{align*}
      \dv\left( V\right)&=2X\left( h \right)\Delta h+2\langle \grad X\left( h \right),\grad h\rangle-\left|\grad h\right|^{2}\dv\left( X \right)\\
      &\qquad-\left\langle\grad \left|\grad h\right|^{2},X\right\rangle\\
      &=2\langle \nabla_{\grad h}X,\grad h\rangle+2\left\langle X,\nabla_{\grad h}\grad h\right\rangle-\left|\grad h\right|^{2}\dv\left( X \right)\\
      &\qquad -2\left\langle\nabla_{X}\grad h,\grad h\right\rangle\\
      &=2\langle \nabla_{\grad h}X,\grad h\rangle-\left|\grad h\right|^{2}\dv\left( X \right)\\
      & \qquad+2\left( \nabla^{2}h\left( X,\grad h \right)-\nabla^{2}h\left( X,\grad h \right) \right)\\
      &=\left|\grad h\right|^{2}\left( 2\left\langle \nabla_{\grad h/\left|\grad h\right|}X,\frac{\grad h}{\left|\grad h\right|}\right\rangle-\dv\left( X \right) \right).
   \end{align*}
   Applying the divergence theorem we get 
   \begin{align*}
      \int_{S}\left\langle V,\partial_n\right\rangle\sigma&=\int_{\mathcal{R}}\left|\grad h\right|^{2}\left( 2\left\langle \nabla_{\grad h/\left|\grad h\right|}X,\frac{\grad h}{\left|\grad h\right|}\right\rangle-\dv\left( X \right) \right)\vol.
   \end{align*}
   Using the definition of $V$ gives $$\left\langle V,\partial_n\right\rangle=-w\left( x \right)\left(\left|\grad_{S}h\right|^{2}-h_n^{2}  \right),$$
   when $x\in S.$ Hence 
\begin{align*}
   \int_{S}&\left( \left|\grad_{S}h\right|^{2}-h_{n}^{2} \right) w\sigma\\
   &=\int_{\mathcal{R}}|\grad h|^{2}\left( \dv\left( w\partial_n \right)-2\left\langle \nabla_{\frac{\grad h}{\left|\grad h\right|}}\left( w\partial_n \right),\frac{\grad h}{\left|\grad h\right|}\right\rangle  \right)\vol.
\end{align*} 
\end{proof}
\begin{remark}
   For Lemma \ref{lem:4:proof} to hold it is not enough for the function $f$ to be Lipschitz. Consider for example the function $f:\mathbf{R}^2\to \mathbf{R}$ is defined by $f\left( x,y \right)=|x|+|y|.$ In this case we have that the level surfaces are squares. Considering the family of harmonic functions $h\left( x,y \right)=e^{nx}\cos\left( ny+\pi/2 \right)$ we get that 
   \begin{equation*}
      \int_{S_1}\left|\grad_{S_1}h\right|^{2}-h_{n}^{2}\sigma_1=-n\sqrt{2}\left( 2\sinh(2n)-2\sin\left( 2n \right) \right)
   \end{equation*}
   and $$\int_{S_1}hh_n\sigma_1=\cosh\left( 2n \right)-1.$$
   Thus there is no $K$ such that $$\int_{S_{1}}\left( \left|\grad_{S_{1}} h\right|^{2}-h_n^{2} \right)\sigma_1\ge -K\int_{S_{1}}hh_n\sigma_1$$ holds for all functions in this family.
\end{remark}
\section{Examples}
Although Theorem \ref{thm:1} is rather technical, it has several novel applications which are explored in this section. As stated in the introduction, we start with an application to geodesic spheres on Riemannian manifolds. In this case, we will use results from comparison geometry to find the functions $M,\,m,\, g$ and $K$ in Theorem \ref{thm:1}. Thereafter we consider level surfaces of $1$-homogeneous functions which cover ellipsoids with constant eccentricity. The distance function for closed lower dimensional spheres will be an example of level surfaces that are not homeomorphic to spheres. Finally, we will show if we have upper and lower estimates on the sectional curvature we have that eigenfunctions of the Laplacian corresponding to positive eigenvalues satisfy the same type of convexity as harmonic functions.

\subsection{Geodesic Spheres}\label{sec:sphere}
Using exponential coordinates centered at a point $p\in \mathcal{M}$ we can introduce polar coordinates in a neighborhood of $p$. Define the \textit{radial distance function} on a normal neighborhood of $p$ by $$r\left( x \right)=\textrm{dist}\left( x,p \right)=\sqrt{x_1^2(x)+\cdots+x_n^2(x)},$$ where $x_i$ are the coordinate functions in the normal neighborhood. In this example we let the function $f$ given in Theorem \ref{thm:1} be $f\left( x \right)=r\left( x \right)$. The level surfaces of $r$ are precisely the geodesic spheres $S_t=r^{-1}\left( t \right)$ of radius $t$. Moreover, the Riemannian metric in polar coordinates can be written as $\mathbf{g}=dr^{2}+g_{t}$ where $g_{t}$ is the induced metric on $S_t$. Let $\inj\left( p \right)$ denote the \textit{injectivity radius} at the point $p$, i.e. the supremum over the radius of all balls centered at $0 \in T_{p}\mathcal{M}$ where the exponential map is injective. Then $r$ is smooth in $B_{\inj\left( p \right)}\left( p \right)\setminus \left\{ p \right\}$.
We will use the notation 
\begin{equation}\label{sin}
   \sin_{\mathcal{K}}\left( t \right)=
   \begin{cases}
      \frac{\sin\left( \sqrt{\mathcal{K}}t \right)}{\sqrt{\mathcal{K}}},&\quad \text{when } \mathcal{K}>0\\
      t, &\quad \text{when } \mathcal{K}=0\\
      \frac{\sinh\left( \sqrt{-\mathcal{K}}t \right)}{\sqrt{-\mathcal{K}}},&\quad \text{when } \mathcal{K}<0,\\
   \end{cases}
   \end{equation}
   $$\cos_{\mathcal{K}}\left( t \right)=\left( \sin_{\mathcal{K}}\left( t \right) \right)',\, \text{and}\,  \cot_{\mathcal{K}}\left( t \right)=\left( \log\left( \sin_{\mathcal{K}}\left( t \right) \right) \right)'=\frac{\cos_{\mathcal{K}}\left( t \right)}{\sin_{\mathcal{K}}\left( t \right)}.$$
\begin{theorem}
   \label{thm:3}
   Assume that $\left( \mathcal{M},\mathbf{g} \right)$ is an $n$-dimensional Riemannian manifold with $p \in \mathcal{M}$ and with sectional curvature $\Sec$satisfying  
   \begin{equation}
   \label{eq:20}
       \kappa |v|^{2}\le \Sec\left( v,v \right)\le \mathcal{K}\left|v\right|^{2},
   \end{equation}
   where $\kappa, \, \mathcal{K}\in \mathbf{R}$ and $v\in T\mathcal{M}$.
   Set $R:=\min\left( \inj\left( p \right),\frac{\pi}{2\sqrt{\mathcal{K}}} \right)$ if $\mathcal{K} > 0$ and $R:=\inj\left( p \right)$ whenever $\mathcal{K} \leq 0$.
   Let $h$ be a harmonic function defined on $B_{R}\left( p \right)$.
   If $r\left( x \right)=\textrm{dist}\left( x,p \right)$ is the radial distance function and $H\left( t \right)=\int_{r^{-1}\left( t \right)}h^{2}\sigma_t$, then 
   \begin{equation}\label{eq:11}
      H'\left( t \right)\ge \left( n-1 \right)\cot_{\mathcal{K}}\left( t \right)H\left( t \right).
\end{equation} Moreover, we have 
\begin{align}\label{eq:10}
   \left( \log H\left( t \right) \right)''&+\left( \cot_{\mathcal{K}}\left( t \right)+\left( n+1 \right)\left( \cot_{\kappa}\left( t \right)-\cot_{\mathcal{K}}\left( t \right) \right) \right)\left( \log H\left( t \right) \right)'\\
   &\ge  -\left( n-1 \right)\mathcal{K}+\left( n-2 \right)\min\left( \mathcal{K},0 \right)-\left( n-1 \right)\left( \mathcal{K}-\kappa \right) ,\nonumber
\end{align}
for every $t\in \left( 0,R \right).$
\end{theorem}
\begin{remark}
   \label{re:2}
   \begin{enumerate}
      \item Note that Equation \eqref{eq:11} implies that $H$ is increasing. When $\mathcal{K}>0$ Inequality \eqref{eq:11} is also valid when $t<\tilde{R}:=\min\left( \inj\left( p \right),\frac{\pi}{\sqrt{\mathcal{K}}} \right).$ However, when $\frac{\pi}{2\sqrt{\mathcal{K}}}<t<\frac{\pi}{\sqrt{\mathcal{K}}}$ the function $\cot_\mathcal{K}\left( t \right)$ is negative. To see that $H$ is not necessarily increasing for values $t>R$ we consider the unit sphere $\mathcal{M} = S^2$ and $h\left( x \right)=1$. In this case, we have precisely that $H'\left( t \right)=\cot\left( t \right)H\left( t \right)$. This shows the necessity of the constraint $R$ since $\cot\left( t \right)H\left( t \right)$ is negative whenever $t>R$. 
      \item Equation \eqref{eq:10} is slightly better than the one presented in \cite[Theorem 2.2 (ii)]{Ma13} whenever $\mathcal{K}>0$ and $\mathcal{K}\not=\kappa$. The Inequality \eqref{eq:11} in \cite{Ma13} is proved with the right hand side equal to \[-\left( n-1 \right)\mathcal{K}-\left( n-1 \right)\left( 1+\frac{n}{2}\left( n-2 \right) \right)\left( \mathcal{K}-\kappa \right)\] instead of our improvement $-\left( n-1 \right)\mathcal{K}-\left( n-1 \right)\left( \mathcal{K}-\kappa \right).$
   \end{enumerate}
\end{remark}
\begin{proof}
   To prove this theorem, we will apply Theorem \ref{thm:1} and use comparison geometry to find $ M,\,m,\,g$ and $K.$ When $\mathcal{K}<0$, we will need to adapt the Theorem \ref{thm:1} slightly, see Remark \ref{re:3}.

   Rauch Comparison Theorem states that the following estimate hold under the sectional curvature bounds given in \eqref{eq:20} 
   \begin{equation}\label{Rauch}
      \cot_{\mathcal{K}}\left( t \right)g_{t}\le \nabla^{2} r\le \cot_{\kappa}\left( t \right)g_{t} \text{ for }t<\tilde{R}.
   \end{equation}
   The proof of Rauch Comparison Theorem can be found in \cite[Theorem 6.4.3]{Pe16} or \cite{Le18}.
Inequality \eqref{Rauch} implies that
$$m\left( t \right)=\left( n-1 \right)\cot_{\mathcal{K}}\left( t \right)\le \Delta r\le \left( n-1 \right)\cot_{\kappa}\left( t \right)=M\left( t \right).$$
To find $g$ we use the following identity, see \cite[Equation (2) p.~ 276]{Pe16}, 
$$\langle \grad \Delta r,\grad r\rangle=-\left|\nabla^{2} r\right|^{2}-\Ric\left( \grad r,\grad r \right),$$
for all functions with $|\grad r|=1.$ 
Using the Rauch Comparison Theorem we obtain $$\left( n-1 \right)\cot_{\mathcal{K}}^{2}\left( t \right)\le \left|\nabla^{2} r\right|^{2}\le \left( n-1 \right)\cot_{\kappa}^{2}\left( t \right).$$ Hence we conclude that
   $$\langle \grad \Delta r,\grad r\rangle\ge -\left( n-1 \right)\cot_{\kappa}^2\left( t \right)-\left( n-1 \right)\mathcal{K}=g\left( t \right).$$

   Next we need to find $K$ which exists by Inequality \eqref{eq:7}. To do this, we will use the following version of Lemma \ref{lem:4:proof}.
\begin{lemma}\label{lem:1}
   Let $\frac{\grad r}{\varphi\left( r\left( x \right) \right)}$ be a smooth vector field, then 
   \begin{align*}
      &\int_{S_{t}}\left( |\grad_{S_t} h|^{2}-h_{n}^{2} \right)\sigma_{t}\\
      &\quad=\varphi\left( t \right) \int_{B_{t}}|\grad h|^{2}\left( \frac{\varphi\left( r\left( x \right) \right)\Delta r-\varphi'\left( r\left( x \right) \right)}{\varphi^{2}\left( r\left( x \right) \right)} \right)\vol\\
      &\quad-2\varphi\left( t \right) \int_{B_{t}}\frac{\nabla^{2}r\left( \grad h,\grad h \right) \varphi\left( r\left( x \right) \right)-\varphi'\left( r\left( x \right) \right)\langle \grad r,\grad h\rangle^{2}}{\varphi\left( r\left( x \right) \right)^{2}}\vol.
\end{align*}
\end{lemma}
\begin{proof}
   Fix $t_0$. Using Lemma \ref{lem:4:proof} with the extension of $\partial_n$ to $B_{t_0}$ be equal to $\frac{\varphi\left( t_0 \right)\grad r}{\varphi\left( r\left( x \right) \right)}$ gives
\begin{align*}
   &\int_{S_{t}}\left( \left|\grad_{S}h\right|^{2}-h_{n}^{2} \right) \sigma_{t}\\
   &=\varphi\left( t_0 \right)\int_{B_{t}}|\grad h|^{2}\left( \dv\left( \frac{\grad r}{\varphi\left( r\left( x \right) \right)} \right)-2\left\langle \nabla_{\frac{\grad h}{\left|\grad h\right|}}\left( \frac{\grad r}{\varphi\left( r\left( x \right) \right)}\right),\frac{\grad h}{\left|\grad h\right|}\right\rangle  \right)\vol.
\end{align*} 
Using the product rule for the divergence and covariant derivative finishes the proof.
\end{proof}

Using Lemma \ref{lem:1} with $\varphi\left( t \right)=\frac{1}{\sin_{\mathcal{K}}\left( t \right)}$ implies 
\begin{align}\label{eq:grad}
      \int_{S_{t}}\left( |\grad_{S_t} h|^{2}-h_n^{2} \right)&\sigma_{t}=\frac{1}{\sin_{\mathcal{K}}\left( t \right)} \int_{B_{t}}\big(|\grad h|^{2}\left( \sin_{\mathcal{K}}\left( s \right)\Delta r+\sin_{\mathcal{K}}\left( s \right)\cot_{\mathcal{K}}\left( s \right) \right)\nonumber\\ 
      &-2\left( \sin_{\mathcal{K}}\left( s \right)\nabla^{2}r\left( \grad h,\grad h \right) -\cot_{\mathcal{K}}\left( s \right)\sin_{\mathcal{K}}\left( s \right)h_{n}^{2} \right)\big)\vol
   \end{align}
   Applying Rauch Comparison Theorem gives
   \begin{align*}
       \int_{S_{t}}\left( |\grad_{S_t} h|^{2}-h_n^{2} \right)\sigma_{t}&\ge\frac{1}{\sin_{\mathcal{K}}\left( t \right)} \int_{B_{t}}|\grad h|^{2}\left( n\cos_{\mathcal{K}}\left( s \right) -2 \sin_{\mathcal{K}}\left( s \right)\cot_{\kappa}\left( s \right)\right)\vol\\
       &\ge 2\left( \cot_{\mathcal{K}}\left( t \right)-\cot_{\kappa}\left( t \right) \right)D\left( t \right)\\
       &\quad+\frac{n-2}{\sin_{\mathcal{K}}\left( t \right)}\int_{B_{t}}\cos_{\mathcal{K}}\left( r\left( x \right) \right)\left|\grad h\right|^{2}\vol
   \end{align*}
   where we have used that $$\sin_{\mathcal{K}}\left( t \right)\left( \cot_{\mathcal{K}}\left( t \right)-\cot_{\kappa}\left( t \right) \right)$$ is decreasing for $t<R$.    
   Using integration by part on the last term together with the observation that $$D\left( t \right)=\frac{H'\left( t \right)-\int_{S_{t}}h^{2}\Delta r\sigma_{t}}{2}$$ 
   we get that
   \begin{align*}
      \int_{S_{t}}\left( |\grad_{S_t} h|^{2}-h_n^{2} \right)\sigma_{t}&\ge 2\left( \cot_{\mathcal{K}}\left( t \right)-\cot_{\kappa}\left( t \right) \right)D\left( t \right)+\left( n-2 \right)\cot_{\mathcal{K}}\left( t \right)D\left( t \right)\\
      &\quad+\left( n-2 \right)\frac{\mathcal{K}}{\sin_{\mathcal{K}}\left( t \right)}\int_{0}^{t}\sin_{\mathcal{K}}\left( s \right)D\left( s \right)ds\\
      &\ge \left( n\cot_{\mathcal{K}}\left( t \right)-2\cot_{\kappa}\left( t \right) \right)D\left( t \right)\\
      &\quad+\frac{\left( n-2 \right)\min(\mathcal{K},0)}{\sin_{\mathcal{K}}\left( t \right)}\int_{0}^{t}\sin_{\mathcal{K}}\left( s \right)\left( \frac{H'\left( s \right)-\int_{S_{s}}h^{2}\Delta r \sigma_s}{2} \right)ds\\
      &\ge \left( n\cot_{\mathcal{K}}\left( t \right)-2\cot_{\kappa}\left( t \right) \right)D\left( t \right)+ \left( n-2 \right)\frac{\min (\mathcal{K},0)}{2}H\left( t \right),
   \end{align*}
   where we have used that $$\left( \cos_{\mathcal{K}}\left( t \right) \right)'=-\mathcal{K}\sin_{\mathcal{K}}\left( t \right).$$
   Setting $$K\left( t \right)=2\cot_{\kappa}\left( t \right)-n\cot_{K}\left( t \right)$$ we finish the case when $\mathcal{K}>0.$ When $\mathcal{K}<0$, we use Remark \ref{re:3} with $$k\left( t \right)=\left( n-2 \right)\frac{\min\left( \mathcal{K},0 \right)}{2}.$$ Hence we have that $$M\left( t \right)+K\left( t \right)=\cot_{\mathcal{K}}\left( t \right)+\left( n+1 \right)\left( \cot_{\kappa}\left( t \right)-\cot_{\mathcal{K}}\left( t \right) \right)$$ and
   \begin{align*}
   g\left( t \right)&+m\left( t \right)\left( M\left( t \right)+K\left( t \right) \right)+2k\left( t \right)=-\left( n-1 \right)\mathcal{K}+\left( n-2 \right)\min\left( \mathcal{K},0 \right)\\
   &-\left( n-1 \right)\cot_{\kappa}^{2}\left( t \right)+\left( n-1 \right)\cot_{\mathcal{K}}\left( t \right)\left( \cot_{\mathcal{K}}\left( t \right)+\left( n+1 \right)\left( \cot_{\kappa}\left( t \right)-\cot_{\mathcal{K}}\left( t \right) \right) \right)\\
   &\ge -\left( n-1 \right)\mathcal{K}+\left( n-2 \right)\min\left( \mathcal{K},0 \right)+\left( n-1 \right)\left( \cot_{\mathcal{K}}^{2}\left( t \right)-\cot_{\kappa}^{2}\left( t \right) \right).
\end{align*}
Using that $$\cot_{\mathcal{K}}^{2}\left( t \right)-\cot_{\kappa}^{2}\left( t \right)\ge \kappa-\mathcal{K},$$ see \cite[p.652]{Ma13},
we get that 
   \begin{align*}
   g\left( t \right)+m\left( t \right)\left( M\left( t \right)+K\left( t \right) \right)+2k\left( t \right)&\ge-\left( n-1 \right)\mathcal{K}+\left( n-2 \right)\min\left( \mathcal{K},0 \right)\\
   &\quad+\left( n-1 \right)\left( \kappa-\mathcal{K} \right).
\end{align*}
\end{proof}

Let us briefly discuss the sharpness of our results in Theorem \ref{thm:3}. Remember that in $\mathbf{R}^2$ the homogeneous harmonic polynomials can be written in polar coordinates as $$h\left( t,\theta \right)=t^{k}\left( a\cos\left( k\theta \right)+b\sin\left( k\theta \right) \right),$$ where $a,b\in \mathbf{R}$. In this case we have that $\mathcal{K}=\kappa=0$ and Theorem \ref{thm:3} becomes $$\left( \log H\left( t \right) \right)''+\frac{1}{t}\left( \log H\left( t \right) \right)'\ge 0.$$ For the homogeneous polynomials we have that the inequality is sharp. Let $\mathcal{K}=\kappa$ and define $\tan_{\mathcal{K}}\left( t \right)=\frac{1}{\cot_{\mathcal{K}}\left( t \right)}$. The equivalent of homogeneous harmonic polynomials for the two dimensional constant curvature spaces are $$h\left( t,\theta \right)=\tan_{\mathcal{K}}\left( \frac{t}{2} \right)^{k}(a\cos\left( k\theta \right)+b\sin\left( k\theta \right)).$$ In this case we have that Theorem \ref{thm:3} becomes 
\begin{equation*}
   \left( \log H\left( t \right) \right)''+\cot_{\mathcal{K}}\left( t \right)\left( \log H\left( t \right) \right)'\ge  -\mathcal{K}, 
\end{equation*}
and again we have that for the functions $h$ we have that the inequality is sharp. When $\mathcal{K}\ge0$ we have that for the constant harmonic function Theorem \ref{thm:3} is sharp for all $n$. In the case when $\mathcal{K}<0$ doing the example of constant harmonic functions would suggest that the inequality could be improved to the right hand side being $-\left( n-1 \right)\mathcal{K}.$ When $\mathcal{K}<0$ and $n\ge 2$ the radial solutions using spherical harmonics can be found in \cite[Proposition 4.2]{Mi75}. However, the solutions are expressed using hypergeometric functions and it is thus no trivial task to see if the result is sharp for these solutions.

\subsection{$1$-Homogeneous Functions}\label{1-hom}
The natural next step from looking at spheres in $\mathbf{R}^n$ is to look at families of surfaces in $\mathbf{R}^n$ where the domains bounded by the surfaces are star convex with respect to the origin. Fix a smooth and compact surface $S \subset \mathbf{R}^n$ such that the origin is not contained in $S$. Moreover, assume that for each point $\mathbf{x} \in S$ the ray $\{t\mathbf{x}: t \geq 0\}$ intersects the surface $S$ precisely once, namely at $t = 1$. The we can unambiguously define the \textit{inside} of $S$ to be the collection of points \[\textrm{Inn}(S) := \bigcup_{\mathbf{x} \in S} \bigcup_{0 \leq t < 1}t\mathbf{x}.\] It is clear from its definition that $\textrm{Inn}(S)$ is star convex with respect to the origin. \par
A function $g:\mathbf{R}^n\to \mathbf{R}$ is called \textit{$k$-homogeneous} for $k\in \mathbf{Z}$ if $g(t\mathbf{x})=t^kg(\mathbf{x})$ for all $t\ge 0.$  
Define $f$ to be $f:\mathbf{R}^{n}  \to [0,\infty)$ by requiring that $f \equiv 1$ on $S$ and \[f(t\mathbf{x}) := t\cdot f(\mathbf{x}) = t, \qquad \mathbf{x} \in S, \, t \ge 0.\] Then $f$ is a \textit{$1$-homogeneous function} since $f(t\mathbf{x}) = tf(\mathbf{x})$ for every $t \ge 0$ and $\mathbf{x} \in \mathbf{R}^{n}$. Given a $1$-homogeneous with smooth compact level surface $f$. Denote by $S_t=f^{-1}(t)$, then $\textrm{Inn}(S_t)$ are star convex with respect to $\mathbf{0}.$
\begin{proposition}\label{prop:2}
   Let $f:\mathbf{R}^n \to [0,\infty)$ be a $1$-homogeneous function which is smooth in $\mathbf{R}^n\setminus\{\mathbf{0}\}$ with compact smooth level surfaces $S_t$. Consider a harmonic function $h:\mathbf{R}^n\to \mathbf{R}$ and set $$H\left( t \right)=\int_{S_{t}}h^{2}\left( \mathbf{x} \right)\left|\grad f\right|\sigma_t.$$ Then the function $H$ satisfies 
\begin{align}\label{eq:1}
   \frac{H''\left( t \right)H\left( t \right)-H'\left( t \right)^{2}+\frac{A}{t}H'\left( t \right)H\left( t \right)}{H\left( t \right)^{2}}\ge -B/t^{2},
\end{align}
where the positive constants $A$ and $B$ only depend on $f.$
\end{proposition}
\begin{proof}
   To apply Theorem \ref{thm:1} we will use the fact that the derivative of a $k$-homogeneous function is a $(k-1)$-homogeneous function. Thus the derivatives of $f$ satisfy $f_{x_i}\left( t\mathbf{x} \right)=f_{x_i}\left( \mathbf{x} \right)$, $f_{x_ix_j}\left( t\mathbf{x} \right)=\frac{1}{t}f_{x_ix_j}\left( \mathbf{x} \right)$ and $f_{x_ix_jx_k}\left( t\mathbf{x} \right)=\frac{1}{t^{2}}f_{x_ix_jx_k}\left( \mathbf{x} \right).$ Note also that all $k$-homogeneous functions are uniquely determined by their restrictions to the unit sphere $S^{n-1}\subset \mathbf{R}^n$. This implies that the estimates we need to satisfy in Theorem \ref{thm:1} are given by taking the minimum or maximum of the derivatives over $S^{n-1}$. 
Hence we can take $m(t)=C_1/t$, $M(t)=C_2/t$ and $g(t)=C_3/t^{2}$ for some constants $C_1, \, C_2$ and $C_3$. 

Fix $t_0 > 0$. To find $K\left( t \right)$ we extend the normal vector field on $S_{t_0}$ to the inside $\textrm{Inn}(S_{t_0})$ by $\partial_n=\frac{f\grad f}{t_0\left|\grad f\right|}$. Then by Equation \eqref{eq:3} we obtain
\begin{align*}
   \int_{S_{t_0}}&\frac{\left( \left|\grad_{S_{t_0}}h\right|^{2}-h_{n}^{2} \right)}{\left|\grad f\right|}\sigma_{t_0}\\
   &=\int_{R_{t_0}}|\grad h|^{2}\dv\left( \frac{f\grad f}{t_0\left|\grad f\right|^{2}} \right)-2\left\langle \nabla_{\grad h}\left( \frac{f\grad f}{t_0\left|\grad f\right|^{2}}\right),\grad h\right\rangle \vol\\
   &\ge \frac{-C_4}{t_0}\int_{R_{t}}\left|\grad h\right|^{2}\vol,
\end{align*} 
for some constant $C_4$ where the last inequality follows from the components of $\frac{f\grad f}{\left|\grad f\right|}$ being $1$-homogeneous in each component. 
Using Theorem \ref{thm:1} we obtain Inequality \eqref{eq:1} with $A=C_2+C_4$ and $B=C_3+C_1C_2+C_1C_4.$
\end{proof}
Note that if $h$ is a homogeneous harmonic functions of degree $k$ then $H$ becomes an $\left( n-1+2k \right)$ homogeneous function. Thus $H\left( t \right)=t^{n-1+2k}H\left( 1 \right)$. Hence 
\begin{align*}
\frac{H''\left( t \right)H\left( t \right)-H'\left( t \right)^{2}+\frac{A}{t}H'\left( t \right)H\left( t \right)}{H\left( t \right)^{2}}&=\left( -\left( n-1+2k \right)+A\left( n-1+2k \right) \right)/t^{2}\\
&=\left( A-1 \right)\left( n-1+2k \right)\ge -B/t^{2}.
\end{align*}
Since this holds for all $t$ we have that $A\ge1.$

We can integrate the inequality in Proposition \ref{prop:2} and get a convexity property for $H.$ Doing this we get the following corollary.
\begin{corollary}
   \label{prop:3}
   \begin{enumerate}
      \item When we have that $A > 1,$ $H$ satisfies $$H\left( t_1 \right)\le \left( \frac{t_0}{t_1} \right)^{\alpha \frac{B}{A-1}}\left( \frac{t_2}{t_1} \right)^{\left( 1-\alpha \right) \frac{B}{A-1}}H\left( t_0 \right)^{\alpha}H\left( t_2 \right)^{1-\alpha},$$
where 
   $$\left( 1-\alpha \right)\left( \frac{t_1}{t_2} \right)^{A-1}+\alpha\left( \frac{t_1}{t_0} \right)^{A-1}=1,\quad t_0\le t_1\le t_2.$$
   In this case, the function $$N_H\left( t \right):=t^{A-1}\left( \frac{tH'\left( t \right)}{H\left( t \right)}+\frac{B}{A-1} \right)$$ is increasing.
\item When $A=1$ we have that
   $$H\left( t_1 \right)\le \exp\left(-\frac{B}{2}\log\left( \frac{t_0}{t_1} \right)\log\left( \frac{t_2}{t_1} \right)\right)H\left( t_0 \right)^{\alpha}H\left( t_2 \right)^{1-\alpha},$$
   where
   $$\left( 1-\alpha \right)\log\left( \frac{t_1}{t_2} \right)+\alpha\log\left( \frac{t_1}{t_0} \right)=0.$$
   In this case, the function $$N_{H}(t):=\frac{tH'\left( t \right)}{H\left( t \right)}+B\log\left( t \right)$$ is increasing. 
   \end{enumerate}
\end{corollary}
\begin{proof}
Assume first that $A>1$. 
By using the integrating factor $t^{A}$ inequality \eqref{eq:1} becomes
\begin{align}\label{eq:2}
   \left( t^A\left( \log H\left( t \right) \right)' +\frac{B}{A-1}t^{A-1}\right)'\ge 0.
\end{align}
Hence the function $$N_H\left( t \right)=t^{A-1}\left( \frac{tH'\left( t \right)}{H\left( t \right)}+\frac{B}{A-1} \right)$$ is increasing. Define $$G\left( t \right)=t^{\frac{B}{A-1}}H\left( t \right).$$ Then $t^{A}\left( \log\left( G\left( t \right) \right) \right)'=\gamma\left( t \right)$ is an increasing function and 
$$\log G\left( t_1 \right)-\log G\left( t_0 \right)=\int_{t_0}^{t_1}\gamma\left( t \right)t^{-A}dt\le \gamma\left( t_1 \right)\frac{t_{0}^{1-A}-t_1^{1-A}}{A-1}.$$
Similarly, 
$$\log G\left( t_2 \right)-\log G\left( t_1 \right)\ge \gamma\left( t_1 \right)\frac{t_{1}^{1-A}-t_2^{1-A}}{A-1}.$$
We also know that $$\alpha\left( t_{0}^{1-A}-t_{1}^{1-A} \right)=\left( 1-\alpha \right)\left( t_{1}^{1-A}-t_{2}^{1-A} \right).$$ This implies the required inequality $G\left( t_{1} \right)\le G\left( t_0 \right)^{\alpha}G\left( t_2 \right)^{1-\alpha}.$


Whenever $A=1$ we obtain through similar computations that $$\left( \log \left( H\left( e^{t} \right) \right)+\frac{t^{2}}{2}B \right)''\ge 0.$$ Since $H\left( e^{t} \right)+\frac{t^{2}}{2}B$ is convex we get 
\begin{align*}
   e^{B\left( \log t_1 \right)^{2}/2}H\left( t_{1} \right)\le  e^{\alpha B\left( \log t_0 \right)^{2}/2+\left(1- \alpha \right) B\left( \log t_0 \right)^{2}/2}H\left( t_{0} \right)^{\alpha}H\left( t_2 \right)^{1-\alpha}.
\end{align*}
Using that $$\left( 1-\alpha \right)\log\left( t_2 \right)=\log\left( t_1 \right)-\alpha\log\left( t_0 \right)$$ and $$\alpha\log\left( t_2 \right)=\log\left( t_1 \right)-\left(1- \alpha \right)\log\left( t_2 \right)$$
we get that 
   $$H\left( t_1 \right)\le \exp\left(-\frac{B}{2}\log\left( \frac{t_0}{t_1} \right)\log\left( \frac{t_2}{t_1} \right)\right)H\left( t_0 \right)^{\alpha}H\left( t_2 \right)^{1-\alpha}.$$
In this case, the function $N_H\left( t \right)=\frac{tH'\left( t \right)}{H\left( t \right)}+B\log\left( t \right)$ is increasing.
\end{proof}

\subsubsection{Ellipsoids with Constant Eccentricity}\label{ell}
We will now specialize to the case of ellipsoids with constant eccentricity. Define the dilation matrix 
$$D=
\begin{bmatrix}
   a_1&0&\dots& 0\\
   0&a_2&\dots&0\\
   \vdots&&\ddots&0\\
   0&\dots&0&a_n
\end{bmatrix},$$
where $0<a_1\le a_2\le\dots\le a_n$. The function $f\left( \mathbf{x} \right)=\left|D^{-1}\mathbf{x}\right|$ is $1$-homogeneous and its level surfaces are ellipsoids in $\mathbf{R}^n$ centered at the origin. We wish to illustrate Proposition \ref{prop:2} and will hence need to find the values $A$ and $B$ in Proposition \ref{prop:2} explicitly. To find $A$ and $B$ we will use Theorem \ref{thm:1}.

A straightforward calculation gives the gradient, Hessian and Laplacian of $f$ as $$\grad f=\frac{D^{-2}\mathbf{x}}{f\left( \mathbf{x} \right)},\qquad \nabla^{2}f=\frac{D^{-2}}{f\left( \mathbf{x} \right)}-\frac{\left( D^{-2}\mathbf{x} \right)\left( D^{-2}\mathbf{x} \right)^{T}}{f^{3}\left( \mathbf{x} \right)}, \qquad \Delta f=\frac{\tr D^{-2}}{f\left( \mathbf{x} \right)}-\frac{\left|D^{-2}\mathbf{x}\right|^{2}}{f^{3}\left( \mathbf{x} \right)}.$$ Hence we obtain the estimates $$\frac{a_1^{2}\tr\left(  D^{-2} \right)-1}{f\left( \mathbf{x} \right)}\le\frac{\Delta f}{\left|\grad f\right|^{2}}\le \frac{a_n^{2}\tr\left(  D^{-2} \right)-1}{f\left( \mathbf{x} \right)}.$$ We may now set $M\left( t \right)=\frac{a_n^{2}\tr\left(  D^{-2} \right)-1}{t}$ and $m\left( t \right)=\frac{a_1^{2}\tr\left(  D^{-2} \right)-1}{t}$, such that $M$ and $m$ are the functions in Equation \eqref{eq:14}.

To find a candidate for $g$ in Equation \eqref{eq:13} we compute that
\begin{align*}
   \left\langle\grad \frac{\Delta f}{\left|\grad f\right|^{2}},\frac{\grad f}{\left|\grad f\right|^{2}} \right\rangle&=\frac{1}{f^{2}\left( \mathbf{x} \right)}+\frac{\tr\left( D^{-2} \right)}{\left|D^{-2}\mathbf{x}\right|^{2}}-\frac{2\tr\left( D^{-2} \right)f^{2}\left( \mathbf{x} \right)\left|D^{-3}\mathbf{x}\right|^{2}}{\left|D^{-2}\mathbf{x}\right|^{6}}\\
&\ge \frac{1+a_{1}^{2}\tr\left( D^{-2} \right)-2\left( a_n^{4}/a_1^{2} \right)\tr\left( D^{-2} \right)}{f^{2}\left( \mathbf{x} \right)}.
\end{align*}
This allows us to set $$g\left( t \right)=\frac{1+a_{1}^{2}\tr\left( D^{-2} \right)-2\left( a_n^{4}/a_1^{2} \right)\tr\left( D^{-2} \right)}{t^{2}}.$$

Next we want to use Lemma \ref{lem:4:proof} to find the function $K$. Fix $t_0$ and extend the unit normal of the ellipsoid $S_{t_{0}}$ to the inside of $S_{t_{0}}$ by $\partial_n=\frac{f\left( \mathbf{x} \right)\grad f}{t_{0}\left|\grad f\right|}$. Then 
\begin{align*}
   \dv\left( \frac{f\left( \mathbf{x} \right)\grad f}{t_0\left|\grad f\right|^{2}} \right)&=\frac{1}{t_{0}}\left( 2-\frac{2f^{2}\left( \mathbf{x} \right)\left|D^{-3}\mathbf{x}\right|^{2}}{\left|D^{-2}\mathbf{x}\right|^{4}}+\frac{ \tr\left( D^{-2} \right)f^{2}\left( \mathbf{x} \right) }{\left|D^{-2}\mathbf{x}\right|^{2}} \right)\\
&\ge \frac{1}{t_0}\left( 2-2\frac{a_n^2}{a_1^2}+a_1^{2}\tr\left( D^{-2} \right) \right).
\end{align*}
Furthermore, we have that 
\begin{align*}
   & \left\langle \nabla_{\grad h}\left( \frac{f\left( \mathbf{x} \right)\grad f}{t_0\left|\grad f\right|^{2}} \right),\grad h\right\rangle\\
   &\quad\quad= \frac{\left|\grad h\right|^{2}}{t_0}\bigg(\frac{\langle \grad f,\frac{\grad h}{|\grad h|}\rangle^{2}}{\left|\grad f\right|^{2}}-\frac{2f\left( \mathbf{x} \right)\nabla^{2}f\left( \frac{\grad h}{\left|\grad h\right|}, \frac{\grad f}{\left|\grad f\right|}\right)\left\langle \frac{\grad h}{\left|\grad h\right|}, \frac{\grad f}{\left|\grad f\right|}\right\rangle }{\left|\grad f\right|^{2}}\\
  &\quad \quad \quad+\frac{f\left( \mathbf{x} \right)}{\left|\grad f\right|^{2}}\nabla^{2}f\left( \frac{\grad h}{\left|\grad h\right|},\frac{\grad h}{\left|\grad h\right|} \right)\bigg)\\
  &\qquad=\frac{\left|\grad h\right|^{2}}{t_0}\bigg(2\left\langle \frac{\grad h}{\left|\grad h\right|},\frac{\grad f}{\left|\grad f\right|}\right\rangle \left\langle \frac{\grad f}{\left|\grad f\right|},\left( I-\frac{D^{-2}}{\left|\grad f\right|^{2}} \right)\frac{\grad h}{\left|\grad h\right|}\right\rangle\\
  &\qquad \quad+\frac{\left\langle \frac{\grad h}{\left|\grad h\right|},D^{-2}\frac{\grad h}{\left|\grad h\right|}\right\rangle}{\left|\grad f\right|^{2}}\bigg) \\
  &\quad\quad\le \frac{\left|\grad h\right|^{2}}{t_0}\left(2+ \frac{a_n^{2}}{a_1^2}-\frac{2a_1^2}{a_n^2} \right).
\end{align*}
Hence using Lemma \ref{lem:4:proof} we get that
\begin{align*}
   \int_{S_t}\frac{\left|\grad_{S_t}h\right|^{2}-h_n^{2}}{\left|\grad f\right|}\sigma_t&\ge -\frac{1}{f\left( \mathbf{x} \right)} \left(2+4\left(\frac{a_n^{2}}{a_1^{2}}-\frac{a_{1}^{2}}{a_n^{2}}  \right)-a_{1}^{2} \tr \left( D^{-2} \right)\right)\int_{S_t}h_nh\sigma_t\\
   &=-K\left( t \right)\int_{S_{t}}hh_n\sigma_t.
\end{align*}
Using Theorem \ref{thm:1} we find the explicit values $$A=1+4\left( \frac{a_n^{2}}{a_1^{2}}-\frac{a_1^{2}}{a_n^{2}} \right)+\left( a_n^{2}-a_1^{2} \right)\tr\left( D^{-2} \right)$$
and 
$$B=4\left( \frac{a_1^2}{a_n^2}-\frac{a_n^2}{a_1^2} \right)+\left(3a_1^2+3a_n^2-2\frac{a_n^{4}}{a_1^{2}} -4\frac{a_1^4}{a_n^2} \right)\tr\left( D^{-2} \right)+a_{1}^{2}\left( a_{n}^{2}-a_1^{2} \right)\tr\left( D^{-2} \right)^{2}.$$
Note that $A=1$ if and only if $f\left( \mathbf{x} \right)=c\left|\mathbf{x}\right|$. In this case, we are integrating over spheres. In all other cases we have $A>1.$

Let $h$ be a harmonic function. Define $h\left( D\mathbf{y} \right)=v\left( \mathbf{y} \right)$. Then $\dv\left(D^{-2}\grad v  \right)=0.$ By using change 
\begin{align*}
   H\left( t \right)&=\int_{S_t}h^{2}\left( \mathbf{x} \right)\frac{\left|D^{-2}\mathbf{x}\right|}{\left|D^{-1}\mathbf{x}\right|}\sigma_t\\
   &=\int_{D^{-1}\left( S_{t} \right)}v^{2}\left( \mathbf{y} \right)\frac{\langle D^{-2}\mathbf{y},\mathbf{y}\rangle}{\left|\mathbf{y}\right|^{2}}\tilde{\sigma_t},
\end{align*}
where $D^{-1}\left( S_{t} \right)$ is the sphere with radius $t$ and $\tilde{\sigma_t}$ is the spherical measure. This is the same measure as was considered in \cite{Ga86}, however they only considered the case the when $\dv\left( A\left( \mathbf{x} \right)\grad u \right)=0$ when $A\left( \mathbf{0} \right)=I$.

\subsection{Example of the distance function of $S^{k}\subset \mathbf{R}^n$}
Let $k<n$ and  $$S^{k}=\left\{ \left( x_1,\dots,x_n \right):\,x_1^{2}+\cdots+x_{k+1}^{2}=1,\, x_{k+2}=\cdots=x_n=0 \right\}.$$ Then the distance from a point $\mathbf{x} \in \mathbf{R}^n$ to the surfaces $S^{k}$ is given by $$f\left( \mathbf{x} \right)=\sqrt{\left( r_{k+1}\left( \mathbf{x} \right)- 1\right)^{2}+x_{k+2}^{2}+\cdots+x_{n}^{2}},$$ where $$r_{k+1}\left( \mathbf{x} \right)=\sqrt{x_1^2+\cdots+x_{k+1}^{2}}.$$
This is a special case of \textit{Fermi coordinates}, see \cite{Ch06}, where the submanifold is $S^{k}$. In the case when $k=0$ the set $S^k$ consists only of two points. In this case, the function $f\left( \mathbf{x} \right)$ is the usual distance function from $\mathbf{x}$ to the nearest of the two points in $S^{0}$. When $k=1$ and $n=3$ the level surfaces $S_{t} = f^{-1}(t)$ for small $t$ are tori. \par
Note that $f$ is not smooth along the set of points $$\left\{\left( 0,\dots,0,x_{k+2},\dots,x_n \right) \, : \, x_{k+2}, \dots, x_n \in \mathbf{R}^n \right\}.$$ Hence we will only consider values in the range of $f$ in $\left[ 0,1-\varepsilon \right)$ for some $0 < \epsilon < 1$. Let $h:f^{-1}\left( \left[ 0,1-\varepsilon \right) \right) \to \mathbf{R}$ be a harmonic function and consider $H\left( t \right)=\int_{S_{t}}h^{2}\left( \mathbf{x} \right)\sigma_t,$ where $S_{t}=f^{-1}\left( t \right)$. Again, we wish to apply Theorem \ref{thm:1}.

The gradient of $f$ is given by $$\grad f =\frac{\mathbf{x}}{f\left( \mathbf{x} \right)}-\frac{\grad r_{k+1}}{f\left( \mathbf{x} \right)}.$$ It follows from a computation that $\left|\grad f\right|=1$ and the Laplacian of $f$ is given by $$\Delta f\left( \mathbf{x} \right)=\frac{n-1}{f\left( \mathbf{x} \right)}-\frac{k/r_{k+1}\left( \mathbf{x} \right)}{f\left( \mathbf{x} \right)}.$$ 
We can similarly compute the gradient of the Laplacian and we find that $$\langle \grad \Delta f, \grad f\rangle=\frac{-\left( n-1 \right)+2k/r_{k+1}\left( \mathbf{x} \right)-k/r_{k+1}\left( \mathbf{x} \right)^{2}}{f^{2}\left( \mathbf{x} \right)}.$$
Assume that $0 < t_0 < 1 - \epsilon$ is fixed. Let $\partial_n=\frac{f\left( \mathbf{x} \right)\grad f}{t_0}$ be the extension of $\partial_n$ to $\textrm{Inn}(S_{t_0})$. If $\mathbf{e}$ is a unit vector then 
\begin{align*}
   \dv\left( \partial_n \right)-2\langle \nabla_{\mathbf{e}}\partial_n,\mathbf{e}\rangle&=\frac{n-k/r_{k+1}\left( \mathbf{x} \right)+2\nabla^{2}r_{k+1}\left(\mathbf{e},\mathbf{e}  \right)-2}{t_0}\\
   &\ge \frac{n-2-k/\varepsilon}{t_0}. 
\end{align*}

In short, if we assume that $\varepsilon<r_{k+1}<2-\varepsilon$ we obtain the expressions $$m\left( t \right)=\frac{n-1-k/\varepsilon}{t}, \qquad M\left( t \right)=\frac{n-1-k/\left( 2-\varepsilon \right)}{t},$$ $$g\left( t \right)=\frac{-\left( n-1 \right)+2k/\left( 2-\varepsilon \right)-k/\varepsilon^{2}}{t^{2}}, \qquad K\left( t \right)=\frac{k/\varepsilon-(n-2)}{t}.$$
Setting $$C=1+k\left( \frac{1}{\varepsilon}-\frac{1}{2-\varepsilon}\right)$$ and $$B=k\left( \frac{\left( 2n-3 \right)(1-\varepsilon)}{\varepsilon\left( 2-\varepsilon \right)} -\frac{2}{\varepsilon^{2}}\right)$$ we get that $H$ satisfy the convexity property 
\begin{align*}
   \frac{H''\left( t \right)H\left( t \right)-H'\left( t \right)^{2}+\frac{C}{t}H\left( t \right)H'\left( t \right)  }{H\left( t \right)^{2}}\ge \frac{B}{t^{2}}.
\end{align*}
Note that this is again an equation on the same form as \eqref{eq:1}. Hence we get the convexity inequality for $H$ given by Corollary \ref{prop:3}. 

\subsection{Non-Positive Eigenvalues of $-\Delta$}\label{sec:eig}
Let $\left( \mathcal{M},\mathbf{g} \right)$ be a non-compact $n$ dimensional Riemannian manifold and sectional curvature bounded by $$\kappa\left|X\right|^{2}\le \Sec\left( X,X \right)\le \mathcal{K}\left|X\right|^{2},$$ where $X$ is a vector field and $\kappa$ and $\mathcal{K}$ are constants. Assume that $u:\mathcal{M}\to \mathbf{R}$ is a solution $\Delta u-k^{2}u=0$. We will show that the spherical $L^2$-norm of $u$ satisfies a convexity property similarly to harmonic functions. \par 
Denote by $S_{1/k} \subset \mathbf{R}^2$ the circle with radius $1/k$. Let $Y\left( \theta \right)$ be the normalized first eigenfunction for $S_{1/k}$ with eigenvalue $-k^{2}$. Extend the function $u\left( x \right)$ to a harmonic function on $\mathcal{M}\times S_{1/k}$ by $h\left( x,\theta \right)=u\left( x \right)Y\left( \theta \right)$, and denote by $\mathcal{H}_{k}$ all harmonic functions created this way. Define the function $f\left( x,\theta \right)=r_{\mathcal{M}}\left( x \right)$, where $r_{\mathcal{M}}$ is the radial distance corresponding to a fixed point $p \in \mathcal{M}$. Then $$S_{t}=f^{-1}\left( t \right)=r_{\mathcal{M}}^{-1}\left( t \right)\times S_{1/k}.$$ In this case, we obtain $$H\left( t \right)=\int_{S_{t}}Y^{2}\left( \theta \right)u^{2}\left( x \right)\sigma_{t}=\int_{r_{\mathcal{M}}^{-1}\left( t \right)\bigcap \mathcal{M}}u^{2}\left( x \right)\tilde{\sigma_{t}},$$ where $\tilde{\sigma_{t}}$ is the measure on the geodesic sphere on $\mathcal{M}$. \par 
Note that we can use Theorem \ref{thm:1} on $\mathcal{H}_{k}$ with the function $f$ as described above. Since $f$ does not depend on $\theta$, to find $m,\,M$ and $g$ we can use Rauch Comparison Theorem on $r_\mathcal{M}$. Using the argumentation found in Section \ref{sec:sphere} we get that $$m\left( t \right)=\left( n-1 \right)\cot_{\mathcal{K}}\left( t \right)\le \Delta f=\Delta r_{\mathcal{M}}\left( x \right)\le \left( n-1 \right)\cot_{\kappa}\left( t \right)=M\left( t \right)$$ and $$\langle \grad \Delta f, \grad f\rangle=\langle \grad \Delta r_{\mathcal{M}}, \grad r_{\mathcal{M}}\rangle\ge g\left( t \right)=-\left( n-1 \right)\cot_{\kappa}^{2}\left( t \right)-\left( n-1 \right)\mathcal{K}.$$ Fix $t_0 > 0$ and let $\partial_n=\frac{\sin_{\mathcal{K}}\left( r\left( x \right) \right)\grad r_{\mathcal{M}}}{\sin_{\mathcal{K}}\left( t_0 \right)}$. Denote by $R_t=f^{-1}\left( \left[ 0,t \right) \right)$, then we have that 
   \begin{align*}
      \int_{S_{t}}&\left( |\grad_{S_t} h|^{2}-h_n^{2} \right)\sigma_{t}=\frac{1}{\sin_{\mathcal{K}}\left( t \right)} \int_{R_{t}}\big(|\grad h|^{2}\left( \sin_{\mathcal{K}}\left( f\left( x \right) \right)\Delta r_{\mathcal{M}}+\cos_{\mathcal{K}}\left( f\left( x \right) \right) \right)\\ 
      &\quad\quad\quad\quad\quad\quad\quad-2\left( \sin_{\mathcal{K}}\left( f\left( x \right) \right)\nabla^{2}r_{\mathcal{M}}\left( \grad h,\grad h \right) +\cos_{\mathcal{K}}\left( f\left( x \right) \right)h_{n}^{2} \right)\big)\vol\\
      &\quad\quad\quad\quad\quad\quad\ge \left( n\cot_{\mathcal{K}}\left( t \right)-2\cot_{\kappa}\left( t \right) \right)D\left( t \right)+\frac{n-2}{2}\min\left( 0,\mathcal{K} \right)H\left( t \right),
   \end{align*}
   by using the same argument as in Section \ref{sec:sphere} and that $h\in \mathcal{H}_k$.
   Using Theorem \ref{thm:1} on the family $\mathcal{H}_k$ we get that $H$ satisfies 
\begin{align*}
   \left( \log H\left( t \right) \right)''&+\left( \cot_{\mathcal{K}}\left( t \right)+\left( n+1 \right)\left( \cot_{\kappa}\left( t \right)-\cot_{\mathcal{K}}\left( t \right) \right) \right)\left( \log H\left( t \right) \right)'\\
   &\ge  -\left( n-1 \right)\mathcal{K}+\left( n-2 \right)\min\left( \mathcal{K},0 \right)-\left( n-1 \right)\left( \mathcal{K}-\kappa \right).
\end{align*}
In short, solutions to $\Delta u=k^{2}u$ satisfy the same convexity estimate as harmonic functions.

\bibliographystyle{my_alpha}
\bibliography{bibliography}
\end{document}